\documentclass[12pt]{article}
\usepackage[margin=1in]{geometry}

\usepackage{amsmath,amsthm,amssymb,amsfonts}
\usepackage{algorithm}
\usepackage{algorithmic}
\usepackage{mathtools}
\usepackage{tcolorbox}
\usepackage{bbm}
\usepackage{bm}

\newcommand{\RR}{\mathbb{R}}
\newcommand{\EE}{\mathbb{E}}

\newcommand{\dist}{\mathbf{dist}}

\usepackage{accents}

\mathchardef\mhyphen="2D % Define a "math hyphen"

\newtheorem{theorem}{Theorem}

\newtheorem{lemma}[theorem]{Lemma}

\newtheorem{proposition}[theorem]{Proposition}
\newtheorem{corollary}[theorem]{Corollary}

\let\originalleft\left
\let\originalright\right
\renewcommand{\left}{\mathopen{}\mathclose\bgroup\originalleft}
\renewcommand{\right}{\aftergroup\egroup\originalright}

\usepackage{todonotes}

\usepackage{mathrsfs}
\begin{document}

	\title{Convergence Rates for Deterministic and Stochastic Subgradient Methods Without Lipschitz Continuity}
	\author{Benjamin Grimmer\footnote{bdg79@cornell.edu; Cornell University, Ithaca NY \newline This material is based upon work supported by the National Science Foundation Graduate Research Fellowship under Grant No. DGE-1650441. This work was done while the author was visiting the Simons Institute for the Theory of Computing. It was partially supported by the DIMACS/Simons Collaboration on Bridging Continuous and Discrete Optimization through NSF grant \#CCF-1740425.}}
	\date{}
	\maketitle

	\begin{abstract}
		We extend the classic convergence rate theory for subgradient methods to apply to non-Lipschitz functions. For the deterministic projected subgradient method, we present a global $O(1/\sqrt{T})$ convergence rate for any convex function which is locally Lipschitz around its minimizers. This approach is based on Shor's classic subgradient analysis and implies generalizations of the standard convergence rates for gradient descent on functions with Lipschitz or H\"older continuous gradients. Further, we show a $O(1/\sqrt{T})$ convergence rate for the stochastic projected subgradient method on convex functions with at most quadratic growth, which improves to $O(1/T)$ under either strong convexity or a weaker quadratic lower bound condition.
	\end{abstract}

	\section{Introduction}
We consider the nonsmooth, convex optimization problem given by
$$ \min_{x\in Q} f(x)$$
for some lower semicontinuous convex function $f \colon \RR^d \rightarrow \RR\cup\{\infty\}$ and closed convex feasible region $Q$. We assume $Q$ lies in the domain of $f$ and that this problem has a nonempty set of minimizers $X^*$ (with minimum value denoted by $f^*$). Further, we assume orthogonal projection onto $Q$ is computationally tractable (which we denote by $P_Q(\cdot)$).

Since $f$ may be nondifferentiable, we weaken the notion of gradients to subgradients. The set of all subgradients at some $x\in Q$ (referred to as the subdifferential) is denoted by
$$\partial f(x) = \{ g\in\RR^d \mid \left(\forall y \in \RR^d\right) \; f(y) \geq f(x) + g^T(y-x)\}.$$
We consider solving this problem via a (potentially stochastic) projected subgradient method. These methods have received much attention lately due to their simplicity and scalability; see~\cite{Bubeck-textbook,Nesterov-introductory}, as well as~\cite{HazanKale2014,Lacoste-Julien2012,Lu2017,Nedic2014,Rakhlin2012} for a sample of more recent works. %Johnstone2017,

Deterministic and stochastic subgradient methods differ in the type of oracle used to access the subdifferential of $f$. For deterministic methods, we consider an oracle $g(x)$, which returns an arbitrary subgradient at $x$. For stochastic methods, we utilize a weaker, random oracle $g(x;\xi)$, which is an unbiased estimator of a subgradient (i.e., $\EE_{\xi\sim D}\ g(x;\xi) \in \partial f(x)$ for some easily sampled distribution $D$).

We analyze two classic subgradient methods, differing in their step size policy.
Given a deterministic oracle, we consider the following normalized subgradient method
\begin{equation} \label{eq:Deterministic-Subgradient-Method}
x_{k+1} := P_Q\left(x_k -\alpha_k \frac{g(x_k)}{\|g(x_k)\|}\right),
\end{equation}
for some positive sequence $\alpha_k$. Note that since $\|g(x_k)\|=0$ only if $x_k$ minimizes $f$, this iteration is well-defined until a minimizer is found.
%We remark that \eqref{eq:Deterministic-Subgradient-Method} is not well-suited for the stochastic setting since $g(x;\xi)/\|g(x;\xi)\|$ is very sensitive to perturbations when $g(x;\xi)$ is small.
Given a stochastic oracle, we consider the following method
\begin{equation} \label{eq:Stochastic-Subgradient-Method}
x_{k+1} := P_Q\left(x_k -\alpha_k g(x_k ; \xi_k)\right),
\end{equation}
for some positive sequence $\alpha_k$ and i.i.d.\ sample sequence $\xi_k\sim D$.

The standard convergence bounds for these methods assume some constant $L>0$ has $\|g(x)\|\leq L$ or $\EE_{\xi} \|g(x;\xi)\|^2 \leq L^2$ for all $x\in Q$. Then after $T>0$ iterations, a point is found with objective gap (in expectation for \eqref{eq:Stochastic-Subgradient-Method}) bounded by
\begin{equation}\label{eq:classic-subgradient-rate}
	f(x) - f^* \leq O\left(\frac{L \|x_0-x^*\|}{\sqrt{T}}\right),
\end{equation}
for any $x^*\in X^*$ under reasonable selection of the step size $\alpha_k$.

The bound $\|g(x)\|\leq L$ for all $x\in Q$ is implied by $f$ being $L$-Lipschitz continuous on some open convex set $U$ containing $Q$ (which is often the assumption made). Uniformly bounding subgradients restricts the classic convergence rates to functions with at most linear growth (at rate $L$).
When $Q$ is bounded, one can invoke a compactness argument to produce a uniform Lipschitz constant. However, such an approach may lead to a large constant heavily dependent on the size of $Q$ (and frankly, lacks the elegance that such a fundamental method deserves).

In stark contrast to these limitations,
early in the development of subgradient methods Shor~\cite{Shor1985} observed that the normalized subgradient method \eqref{eq:Deterministic-Subgradient-Method} enjoys some form of convergence guarantee for any convex function with a nonempty set of minimizers. Shor showed for any minimizer $x^*\in X^*$: some $0\leq k\leq T$ has either $x_k\in X^*$ or
$$ \left(\frac{g(x_k)}{\|g(x_k)\|}\right)^T(x_k - x^*) \leq O\left(\frac{\|x_0-x^*\|}{\sqrt{T}}\right),$$
under reasonable selection of the step size $\alpha_k$.
Thus for any convex function, the subgradient method has convergence in terms of this inner product value (which convexity implies is always nonnegative). This quantity can be interpreted as the distance from the hyperplane $\{x\ |\ g(x_k)^T(x-x_k)=0\}$ to $x^*$. By driving this distance to zero via proper selection of $\alpha_k$, Shor characterized the asymptotic convergence of \eqref{eq:Deterministic-Subgradient-Method}.

There is a substantial discrepancy in generality between the standard convergence bound \eqref{eq:classic-subgradient-rate} and Shor's result.
In this paper, we address this for both deterministic and stochastic subgradient methods. The remainder of this section formally states our generalized convergence rate bounds. For the deterministic case our bounds follow directly from Shor's result, while the stochastic case requires an alternative approach. Then Section \ref{sec:examples} applies these bounds to a few common problem classes outside the scope of uniform Lipschitz continuity. Finally, our convergence analysis is presented in Section~\ref{sec:proofs} and an extension on our model is discussed in Section~\ref{sec:extensions}. %For the deterministic method \eqref{eq:Deterministic-Subgradient-Method}, we show a $O(1/\sqrt{T})$ rate for any convex function with an upper bound of the form $f(x) \leq f(x^*) + \mathcal{D}(\|x-x^*\|)$ for any nonnegative function $\mathcal{D}\colon \RR_{+} \rightarrow \RR_{+}$. For the stochastic method \eqref{eq:Stochastic-Subgradient-Method}, we show a $O(1/\sqrt{T})$ rate for any convex function with at most quadratic growth.

\subsection{Extended Deterministic Convergence Bounds}
Shor's convergence guarantees for general convex functions will serve as the basis of our objective gap convergence rates for the subgradient method \eqref{eq:Deterministic-Subgradient-Method} without assuming uniform Lipschitz continuity. Formally, Shor~\cite{Shor1985} showed the following general guarantee for any sequence of step sizes $\alpha_k$ (for completeness, an elementary proof is provided in Section~\ref{sec:proofs}).
\begin{theorem}[Shor's Hyperplane Distance Convergence]\label{thm:shor}
	Consider any convex $f$ and fix some $x^*\in X^*$. The iteration \eqref{eq:Deterministic-Subgradient-Method} has some $0\leq k\leq T$ with either $x_{k}\in X^*$ or
	\begin{equation} \label{eq:shor-bound}
		\left(\frac{g(x_k)}{\|g(x_k)\|}\right)^T(x_k - x^*) \leq \frac{\|x_0-x^*\|^2 + \sum_{k=0}^T \alpha_k^2}{2\sum_{k=0}^T \alpha_k}.
	\end{equation}
\end{theorem}

The classic objective gap convergence of the subgradient method follows as a simple consequence of this. Indeed, the convexity of $f$ and classically assumed bound $\|g(x_k)\|\leq L$ together with \eqref{eq:shor-bound} imply
$$ \min_{k=0\dots T}\left\{\frac{f(x_k) - f^*}{L}\right\} \leq \frac{\|x_0-x^*\|^2 + \sum_{k=0}^T \alpha_k^2}{2\sum_{k=0}^T \alpha_k}.$$
For example, taking $\alpha_k = \|x_0-x^*\|/\sqrt{T+1}$ for any fixed $T>0$ gives \eqref{eq:classic-subgradient-rate}.

Here Lipschitz continuity enabled us to convert a bound on ``hyperplane distance to a minimizer'' into a bound on the objective gap. Our extended convergence bounds for the deterministic subgradient method follow from observing that much more general assumptions than uniform Lipschitz continuity suffice to provide such a conversion.
In particular, we assume there is an upper bound on $f$ of the form
\begin{equation}\label{eq:deterministic-generalization}
	f(x) - f^* \leq \mathcal{D}(\|x - x^*\|), \hskip0.5cm (\forall x\in Q)
\end{equation}
for some fixed $x^*\in X^*$ and nondecreasing nonnegative function $\mathcal{D}\colon \RR_{+} \rightarrow \RR_{+}\cup\{\infty\}$. 
In this case, we find the following objective gap convergence guarantee.
\begin{theorem}[Extended Deterministic Rate] \label{thm:extended-deterministic-subgradient-rate}
	Consider any convex $f$ satisfying \eqref{eq:deterministic-generalization}. The iteration \eqref{eq:Deterministic-Subgradient-Method} satisfies
	$$ \min_{k=0\dots T}\left\{ f(x_k) - f^*\right\} \leq \mathcal{D}\left(\frac{\|x_0-x^*\|^2 + \sum_{k=0}^T \alpha_k^2}{2\sum_{k=0}^T \alpha_k}\right).$$
	For example, under the constant step size $\alpha_k = \|x_0-x^*\|/\sqrt{T+1}$, the iteration \eqref{eq:Deterministic-Subgradient-Method} satisfies
	$$ \min_{k=0\dots T}\left\{ f(x_k) - f^*\right\} \leq \mathcal{D}\left(\frac{\|x_0-x^*\|}{\sqrt{T+1}}\right).$$
\end{theorem}

Note that any $L$-Lipschitz continuous function satisfies this growth bound with $\mathcal{D}(t) = Lt$. Thus we immediately recover the standard $L\|x_0-x^*\|/\sqrt{T}$ convergence rate. However, using growth bounds allows us to apply our convergence guarantees to many problems outside the scope of uniform Lipschitz continuity. 

Theorem~\ref{thm:extended-deterministic-subgradient-rate} also implies the classic convergence rate for gradient descent on differentiable functions with an $L$-Lipschitz continuous gradient of $O(L\|x_0-x^*\|^2/T)$ \cite{Nesterov-introductory}. Any such function has growth bounded by $\mathcal{D}(t) = Lt^2/2$ on $Q=\RR^d$ (see Lemma~\ref{lem:Lipschitz-Gradient-To-Steepness}). Then a convergence rate immediately follows from Theorem~\ref{thm:extended-deterministic-subgradient-rate} (for simplicity, we consider constant step size).
\begin{corollary}[Generalizing Gradient Descent's Convergence] \label{cor:generalized-gradient-descent}
	Consider any convex function $f$ satisfying \eqref{eq:deterministic-generalization} with $\mathcal{D}(t) = Lt^2/2$. Then under the constant step size $\alpha_k = \|x_0-x^*\|/\sqrt{T+1}$, the iteration \eqref{eq:Deterministic-Subgradient-Method} satisfies
	\begin{equation*} 
	\min_{k=0 \dots T}\left\{f(x_k)-f^*\right\} \leq \frac{L \|x_0-x^*\|^2}{2(T+1)}.
	\end{equation*}
\end{corollary}
Thus a convergence rate of $O(L \|x_0-x^*\|^2/T)$ can be attained without any mention of smoothness or differentiability. In Section~\ref{sec:examples}, we provide a similar growth bound and thus objective gap convergence for any function with a H\"older continuous gradient, which also parallels the standard rate for gradient descent.
In general, for any problem with $\lim_{t\rightarrow 0^+}\mathcal{D}(t)/t = 0$, Theorem~\ref{thm:extended-deterministic-subgradient-rate} gives convergence at a rate of $o(1/\sqrt{T})$.

Suppose that $\mathcal{D}(t)/t$ is finite in some neighborhood of $0$ (as is the case for any $f$ that is locally Lipschitz around $x^*$).  Then simple limiting arguments give the following eventual convergence rate of \eqref{eq:Deterministic-Subgradient-Method} based on Theorem~\ref{thm:extended-deterministic-subgradient-rate}: For any $\epsilon>0$, there exists $T_0>0$, such that all $T>T_0$ have
\begin{equation*}
\mathcal{D}\left(\frac{\|x_0-x^*\|}{\sqrt{T+1}}\right) \leq \left(\limsup_{t\rightarrow 0^+}\frac{D(t)}{t} +\epsilon\right)\frac{\|x_0-x^*\|}{\sqrt{T+1}}.
\end{equation*}
As a result, the asymptotic convergence rate of \eqref{eq:Deterministic-Subgradient-Method} is determined entirely by the rate of growth of $f$ around its minimizers, and conversely, steepness far from optimality plays no role in the asymptotic behavior. %This contrasts the classic theory where large subgradients anywhere hurt the quality of the convergence bound everywhere.

\subsection{Extended Stochastic Convergence Bounds}
Now we turn our attention to giving more general convergence bounds for the stochastic subgradient method. This is harder as we can no longer leverage Shor's result, since normalizing stochastic subgradients may introduce bias or may not be well-defined if $g(x_k;\xi)=0$. As a result, we need a different approach to generalizing the standard stochastic assumptions.

We begin by reviewing the standard convergence results for this method.
\begin{theorem}[Classic Stochastic Rate] \label{thm:classic-stochastic-subgradient-rate}
	Consider any convex function $f$ and stochastic subgradient oracle satisfying $\EE_{\xi} \|g(x;\xi)\|^2 \leq L^2$ for all $x\in Q$. Fix some $x^*\in X^*$. Then for any positive sequence $\alpha_k$, the iteration \eqref{eq:Stochastic-Subgradient-Method} satisfies
	\begin{equation*}
	\EE_{\xi_{0 \dots T}}\left[f\left(\frac{\sum_{k=0}^{T}\alpha_kx_k}{\sum_{k=0}^{T}\alpha_k}\right)-f^*\right] \leq \frac{\|x_0-x^*\|^2 + L^2\sum_{k=0}^{T} \alpha_k^2}{2\sum_{k=0}^{T} \alpha_k}.
	\end{equation*}	
	For example, under the constant step size $\alpha_k = \|x_0-x^*\|/L\sqrt{T+1}$, the iteration \eqref{eq:Stochastic-Subgradient-Method} satisfies
	\begin{equation*}
	\EE_{\xi_{0 \dots T}}\left[f\left(\frac{1}{T+1}\sum_{k=0}^{T}x_k\right)-f^*\right]  \leq \frac{L\|x_0-x^*\|}{\sqrt{T+1}}.
	\end{equation*}
\end{theorem}

We say $f$ is $\mu$-strongly convex on $Q$ for some $\mu> 0$ if for every $x\in Q$ and $g\in \partial f(x)$,
$$ f(y) \geq f(x) + g^T(y-x) + \frac{\mu}{2}\|y-x\|^2 \hskip0.5cm \left(\forall y\in Q\right).$$
Under this condition, the convergence of \eqref{eq:Stochastic-Subgradient-Method} can be improved to $O(1/T)$ \cite{HazanKale2014,Lacoste-Julien2012,Rakhlin2012}. Below, we present one such bound from% from Lacoste-Julien, Schmidt, and Bach
~\cite{Lacoste-Julien2012}.
\begin{theorem}[Classic Strongly Convex Stochastic Rate]\label{thm:classic-strong-stochastic-subgradient-rate}
	Consider any $\mu$-strongly convex function $f$ and stochastic subgradient oracle satisfying $\EE_{\xi} \|g(x;\xi)\|^2 \leq L^2$ for all $x\in Q$. Then for the decreasing sequence of step sizes $\alpha_k = 2/\mu(k+2)$, the iteration \eqref{eq:Stochastic-Subgradient-Method} satisfies
	\begin{equation*}
	\EE_{\xi_{0 \dots T}}\left[f\left(\frac{2}{(T+1)(T+2)}\sum_{k=0}^{T}(k+1)x_k\right)-f^*\right]  \leq \frac{2L^2}{\mu(T+2)}.
	\end{equation*}
\end{theorem}
We remark that Lipschitz continuity and strong convexity are fundamentally at odds. Lipschitz continuity allows at most linear growth while strong convexity requires quadratic growth. The only way both can occur is when $Q$ is bounded.

The standard analysis assumes that $\EE_{\xi}\|g(x;\xi)\|^2$ is uniformly bounded by some $L^2>0$. We generalize this by allowing the expectation to be larger when the objective gap at $x$ is large as well. In particular, we assume a bound of the form
\begin{equation}\label{eq:stochastic-generalization}
	\EE_{\xi}\|g(x;\xi)\|^2 \leq L^2_0 + L_1(f(x)-f^*)
\end{equation}
for some constants $L_0,L_1\geq 0$.
When $L_1$ equals zero, this is exactly the classic model. When $L_1$ is positive, this model allows functions with up to quadratic growth. (To see this, suppose the subgradient oracle is deterministic. Then \eqref{eq:stochastic-generalization} corresponds to a differential inequality of the form $f'(x) \leq \sqrt{L_1 (f(x)-f^*) + L_0^2}$, which has a simple quadratic solution. This interpretation is formalized in Section~\ref{subsec:quadratic-growth-interpretation}.)

The additional generality allowed by \eqref{eq:stochastic-generalization} is important for two reasons. First, it allows us to consider many classic problems which fundamentally have quadratic growth (for example, any quadratically regularized problem, like training a support vector machine, which is considered in Section~\ref{subsec:quadratic-regularization}). Secondly, this model allows us to avoid the inherent conflict in Theorem~\ref{thm:classic-strong-stochastic-subgradient-rate} between Lipschitz continuity and strong convexity since a function can globally satisfy both \eqref{eq:stochastic-generalization} and strong convexity.

Based on this generalization of Lipschitz continuity, we have the following guarantees for convex and strongly convex problems.
\begin{theorem}[Extended Stochastic Rate]\label{thm:extended-stochastic-subgradient-rate}
	Consider any convex function $f$ and stochastic subgradient oracle satisfying \eqref{eq:stochastic-generalization}. Fix some $x^*\in X^*$. Then for any positive sequence $\alpha_k$ with $L_1\alpha_k < 2$, the iteration \eqref{eq:Stochastic-Subgradient-Method} satisfies
	\begin{equation*}
	\EE_{\xi_{0 \dots T}}\left[f\left(\frac{\sum_{k=0}^{T}\alpha_k(2-L_1\alpha_k)x_k}{\sum_{k=0}^{T}\alpha_k(2-L_1\alpha_k)}\right)-f^*\right] \leq \frac{\|x_0-x^*\|^2 + L_0^2\sum_{k=0}^{T} \alpha_k^2}{\sum_{k=0}^{T} \alpha_k(2-L_1\alpha_k)}.
	\end{equation*}	
	For example, under the constant step size $\alpha_k = \|x_0-x^*\|/L_0\sqrt{T+1}$, the iteration \eqref{eq:Stochastic-Subgradient-Method} satisfies
	\begin{equation*}
	\EE_{\xi_{0 \dots T}}\left[f\left(\frac{1}{T+1}\sum_{k=0}^{T}x_k\right)-f^*\right]  \leq \frac{L_0\|x_0-x^*\|}{\sqrt{T+1}} \cdot \frac{2}{2-L_1\alpha_k},
	\end{equation*}
	provided $T$ is large enough to have $L_1\alpha_k < 2$.
\end{theorem}
\begin{theorem}[Extended Strongly Convex Stochastic Rate\footnote{A predecessor of Theorem~\ref{thm:extended-strong-stochastic-subgradient-rate} was given by Davis and Grimmer in Proposition~3.2 of~\cite{DavisGrimmer2017}, where a $O(\log(T)/T)$ convergence rate was shown for certain non-Lipschitz strongly convex problems.}]\label{thm:extended-strong-stochastic-subgradient-rate}
	Consider any $\mu$-strongly convex function $f$ and stochastic subgradient oracle satisfying \eqref{eq:stochastic-generalization}. Fix some $x^*\in X^*$. Then for the decreasing sequence of step sizes $$\alpha_k = \frac{2}{\mu(k+2) + \frac{L_1^2}{\mu(k+1)}},$$ the iteration \eqref{eq:Stochastic-Subgradient-Method} satisfies
	\begin{equation*}
	\EE_{\xi_{0 \dots T}}\left[f\left(\frac{\sum_{k=0}^{T}(k+1)(2-L_1\alpha_k)x_k}{\sum_{k=0}^{T}(k+1)(2-L_1\alpha_k)}\right)-f^*\right]  \leq \frac{2L_0^2(T+1) + L_1^2\|x_0-x^*\|^2/2}{\mu\sum_{k=0}^{T}(k+1)(2-L_1\alpha_k)}.
	\end{equation*}
	The following simpler averaging gives a bound weakened roughly by a factor of two:
	\begin{equation*}
	\EE_{\xi_{0 \dots T}}\left[f\left(\frac{2}{(T+1)(T+2)}\sum_{k=0}^{T}(k+1)x_k\right)-f^*\right]  \leq \frac{4L_0^2}{\mu(T+2)} + \frac{L_1^2\|x_0-x^*\|^2}{\mu(T+1)(T+2)}.
	\end{equation*}
\end{theorem}

\subsection{Related Works}
Recently, Renegar~\cite{Renegar2016} introduced a novel framework that allows first-order methods to be applied to general (non-Lipschitz) convex optimization problems via a radial transformation. Based on this framework, Grimmer~\cite{Grimmer2017-radial-subgradient} showed a simple radial subgradient method has convergence paralleling the classic $O(1/\sqrt{T})$ rate without assuming Lipschitz continuity. This algorithm is applied to a transformed version of the original problem and replaces orthogonal projection by a line search at each iteration.

Lu~\cite{Lu2017} analyzes an interesting subgradient-type method (which is a variation of mirror descent) for non-Lipschitz problems that is customized for a particular problem via a reference function. This approach gives convergence guarantees for both deterministic and stochastic problems based on a relative-continuity constant instead of a uniform Lipschitz constant.

Although the works of Renegar~\cite{Renegar2016}, Grimmer~\cite{Grimmer2017-radial-subgradient}, and Lu~\cite{Lu2017} give convergence rates for specialized subgradient methods without assuming Lipschitz continuity, objective gap guarantees for the classic subgradient methods \eqref{eq:Deterministic-Subgradient-Method} and \eqref{eq:Stochastic-Subgradient-Method}, such as the ones in the present paper, have been missing prior to our work.

	\section{Applications of Our Extended Convergence Bounds} \label{sec:examples}
In this section, we apply our convergence bounds to a variety of problems outside the scope of the traditional theory based on uniform Lipschitz constants.

\subsection{Smooth Optimization}
The standard analysis of gradient descent in smooth optimization assumes the gradient of the objective function is uniformly Lipschitz continuous, or more generally, uniformly H\"older continuous. A differentiable function $f$ has $(L,v)$-H\"older continuous gradient on $\RR^d$ for some $L>0$ and $v\in(0,1]$ if for all $x,y\in\RR^d$
$$ \|\nabla f(y) -\nabla f(x)\| \leq L\|y-x\|^v.$$
Note this is exactly Lipschitz continuity of the gradient when $v=1$.
Below, we state a simple bound on the growth $\mathcal{D}(t)$ of any such function.
\begin{lemma} \label{lem:Lipschitz-Gradient-To-Steepness}
	Consider any $f\in C^1$ with a $(L,v)$-H\"older continuous gradient on $\RR^d$ and any minimizer $x^*\in X^*$. Then
	$$ f(x) - f(x^*) \leq \frac{L}{v+1}\|x-x^*\|^{v+1} \hskip0.5cm (\forall x\in \RR^d).$$
\end{lemma}
\begin{proof}
	Since $\nabla f(x^*)=0$, the bound follows directly as
	\begin{align*}
	f(x) & = f(x^*) + \int_{0}^{1}\nabla f(x^*+t(x-x^*))^T(x-x^*)\ dt\\
%	& = f(x) + \nabla f(x)^T(y-x) + \int_{0}^{1}(\nabla f(x+t(y-x)) - \nabla f(x))^T(y-x)\ dt\\
	& \leq f(x^*) + \nabla f(x^*)^T(x-x^*) + \int_{0}^{1} Lt^v\|x-x^*\|^{v+1}\ dt\\
	& = f(x^*) + \frac{L}{v+1}\|x-x^*\|^{v+1}.\qedhere
	\end{align*}
\end{proof}
This lemma with $v=1$ implies any function with an $L$-Lipschitz gradient has growth bounded by $\mathcal{D}(t)=Lt^2/2$. Then Theorem~\ref{thm:extended-deterministic-subgradient-rate} gives our generalization of the classic gradient descent convergence rate claimed in Corollary~\ref{cor:generalized-gradient-descent}. Further, for any function with a H\"olderian gradient, we find the following $O(1/T^{(v+1)/2})$ convergence rate.
\begin{corollary}[Generalizing H\"olderian Gradient Descent's Convergence] \label{cor:generalized-holderian-gradient-descent}
	Consider any convex function $f$ satisfying \eqref{eq:deterministic-generalization} with $\mathcal{D}(t) = Lt^{v+1}/(v+1)$. Then under the constant step size $\alpha_k = \|x_0-x^*\|/\sqrt{T+1}$, the iteration \eqref{eq:Deterministic-Subgradient-Method} satisfies
	\begin{equation*} 
	\min_{k=0 \dots T}\left\{f(x_k)-f^*\right\} \leq \frac{L \|x_0-x^*\|^{v+1}}{(v+1)(T+1)^{(v+1)/2}}.
	\end{equation*}
\end{corollary}

\subsection{Additive Composite Optimization}
Often problems arise where the objective is to minimize a sum of smooth and nonsmooth functions. We consider the following general formulation of this problem
$$ \min_{x\in \RR^d} f(x) := \Phi(x) + h(x),$$
for any differentiable convex function $\Phi$ with $(L_\Phi, v)$-H\"olderian gradient and any $L_h$-Lipschitz continuous convex function $h$. Such problems occur when regularizing smooth optimization problems, where $h$ would be the sum of one or more nonsmooth regularizers (for example, $\|\cdot\|_1$ to induce sparsity).

Additive composite problems can be solved by prox-gradient or splitting methods, which solve a subproblem based on $h$ each iteration. However, this limits these methods to problems where $h$ is relatively simple.
The subgradient method avoids this limitation by only requiring the computation of a subgradient of $f$ each iteration, which is given by $\partial f(x) = \nabla \Phi(x) + \partial h(x)$. The classic convergence theory fails to give any guarantees for this problem since $f$ may be non-Lipschitz. In contrast, we find this problem class has a simple growth bound from which guarantees for the classic subgradient method directly follow.
\begin{lemma} \label{lem:additive-composite}
	Consider any $\Phi\in C^1$ with a $(L_\Phi,v)$-H\"older continuous gradient on $\RR^d$, $L_h$-Lipschitz continuous $h$ on $\RR^d$, and any minimizer $x^*\in X^*$. Then
	$$ f(x) - f(x^*) \leq \frac{L_\Phi}{v+1}\|x-x^*\|^{v+1} +2L_h\|x-x^*\|\hskip0.5cm (\forall x\in \RR^d).$$
\end{lemma}
\begin{proof}
	From the first-order optimality conditions of $f$, we know $g^*:=-\nabla \Phi(x^*)\in\partial h(x^*)$. Define the following lower bound on $f(x)$
	$$l(x) := \Phi(x) + h(x^*) + g^{*T}(x-x^*).$$
	Notice that $f(x)$ and $l(x)$ both minimize at $x^*$ with $f(x^*) = l(x^*)$. Since $l(x)$ has a $(L_\Phi,v)$-H\"older continuous gradient, Lemma~\ref{lem:Lipschitz-Gradient-To-Steepness} implies for any $x\in\RR^d$,
	$$l(x)-l(x^*)\leq \frac{L_\Phi}{v+1}\|x-x^*\|^{v+1}.$$
	The Lipschitz continuity of $h$ implies
	$$l(x) =\Phi(x) +h(x^*)+g^{*T}(x-x^*)\geq \Phi(x) +(h(x)-L_h\|x-x^*\|)-L_h\|x-x^*\|.$$
	Combining these two inequalities completes the proof.
\end{proof}
Plugging $\mathcal{D}(t) = L_\Phi t^{v+1}/(v+1) + 2L_ht$ into Theorem~\ref{thm:extended-deterministic-subgradient-rate} immediately gives the following $O(1/\sqrt{T})$ convergence rate (for simplicity, we state the bound for constant step size).
\begin{corollary}[Additive Composite Convergence]
Consider the deterministic subgradient oracle $\nabla \Phi(x) + g_h(x)$. Then under the constant step size $\alpha_k = \|x_0-x^*\|/\sqrt{T+1}$, the iteration \eqref{eq:Deterministic-Subgradient-Method} satisfies
$$\min_{k=0\dots T}\{f(x_k)-f^*\} \leq  \frac{L_{\Phi} \|x_0-x^*\|^{v+1}}{(v+1)(T+1)^{(v+1)/2}} + \frac{2L_h\|x_0-x^*\|}{\sqrt{T+1}}.$$
\end{corollary}
The first term in this rate exactly matches the convergence rate on functions with H\"olderian gradient like $\Phi$ (see Corollary~\ref{cor:generalized-holderian-gradient-descent}). Further, up to a factor of two, the second term matches the convergence rate on Lipschitz continuous functions like $h$ (see \eqref{eq:classic-subgradient-rate}). Thus the subgradient method on $\Phi(x) + h(x)$ has convergence guarantees no worse than those of the subgradient method on $\Phi(x)$ or $h(x)$ separately.

\subsection{Quadratically Regularized, Stochastic Optimization}\label{subsec:quadratic-regularization}
Another common class of optimization problems result from adding a quadratic regularization term $(\lambda/2)\|x\|^2$ to the objective function, for some parameter $\lambda>0$. Consider solving
$$\min_{x\in \RR^d} f(x) := h(x)+\frac{\lambda}{2}\|x\|^2$$
for any Lipschitz continuous convex function $h$.	
Suppose we have a stochastic subgradient oracle for $h$ denoted by $\EE_{\xi}\ g_h(x;\xi) \in \partial h(x)$ satisfying $\EE_{\xi} \|g_h(x;\xi)\|^2\leq L^2$. Although the function $h$ and its stochastic oracle meet the necessary conditions for the classic theory to be applied, the addition of a quadratic term violates uniform Lipschitz continuity. Simple arguments give a subgradient norm bound like \eqref{eq:stochastic-generalization} and the following $O(1/T)$ convergence rate. 
\begin{corollary}[Quadratically Regularized Convergence]\label{cor:quadratic-regularization-rate}
	Consider the decreasing step sizes
	$$\alpha_k = \frac{2}{\lambda(k+2) + \frac{36\lambda}{k+1}}$$ and stochastic subgradient oracle $g_h(x ;\xi) + \lambda x$. Fix some $x^*\in X^*$. The iteration \eqref{eq:Stochastic-Subgradient-Method} satisfies
	\begin{equation*}
	\EE_{\xi_{0 \dots T}}\left[f\left(\frac{2}{(T+1)(T+2)}\sum_{k=0}^{T}(k+1)x_k\right)-f^*\right]  \leq \frac{24L^2}{\lambda(T+2)} + \frac{36\lambda \|x_0-x^*\|^2}{(T+1)(T+2)}.
	\end{equation*}
\end{corollary}
\begin{proof}
	Consider any $x^*\in X^*$ and $g^*:= -\lambda x^*\in\partial h(x^*)$ (this inclusion follows from the first-order optimality conditions for $x^*$). Define the following lower bound on $f(x)$
	$$l(x) := h(x^*) + g^{*T}(x-x^*) + \frac{\lambda}{2}\|x\|^2.$$
	Notice that $f(x)$ and $l(x)$ both minimize at $x^*$ with $f(x^*) = l(x^*)$. Since $l$ is a quadratic centered at $x^*$, we can bound the size of $\nabla l(x)=g^* +\lambda x$ for any $x\in\RR^d$ as
	$$\|\nabla l(x)\|^2 =\lambda\|x-x^*\|^2 = 2\lambda (l(x) - l(x^*)) \leq 2\lambda (f(x) - f(x^*)).$$
	Applying Jensen's inequality and the assumed subgradient bound implies
	\begin{align*}
	\EE_{\xi}\|g_h(x;\xi) + \lambda x\|^2 & = \EE_{\xi}\|g_h(x;\xi) - g^* + g^* + \lambda x\|^2\\
	& \leq 3\EE_{\xi}\|g_h(x;\xi)\|^2 + 3\|g^*\|^2 + 3\|g^*+\lambda x\|^2\\
	& \leq 6L^2 + 6\lambda(f(x)-f(x^*)).
	\end{align*}
	Noting $f$ is $\lambda$-strongly convex, our bound follows from Theorem~\ref{thm:extended-strong-stochastic-subgradient-rate}.	
	%	Note that this problem is $\lambda$-strongly convex. Thus $\delta(x) \geq \frac{\lambda}{2}\|x-x^*\|^2$. Then we can bound the expected size of our subgradient oracle $g(x;\xi) = g_h(x;\xi) + \lambda x$ by
	%	\begin{align*}
	%	\EE_{\xi}\|g(x;\xi)\|^2 & = \EE_{\xi}\|g_h(x;\xi) + \lambda x^* - \lambda x^* + \lambda x\|^2\\
	%	& \leq 3L^2 + 3\lambda^2\|x^*\|^2 + 3\lambda^2\|x-x^*\|^2\\
	%	& \leq 3L^2 + 3\lambda^2\|x^*\|^2 + 6\lambda\delta(x),
	%	\end{align*}
	%	where the first inequality uses Jensen's inequality and the second inequality uses strong convexity. Thus our stochastic oracle is $\sqrt{L_1t+L_0^2}$-steep for $L_1 = 6\lambda$ and $L_0^2 = 3L^2 + 3\lambda^2\|x^*\|^2$. Then the convergence bound follows from Theorem~\ref{thm:extended-strong-stochastic-subgradient-rate}.
\end{proof}

One common example of a problem of the form $h(x) + (\lambda/2)\|x\|^2$ is training a Support Vector Machine (SVM). Suppose one has $n$ data points each with a feature vector $w_i\in\RR^d$ and label $y_i\in\{-1,1\}$. Then one trains a model $x\in \RR^d$ for some parameter $\lambda>0$ by solving
$$ \min_{x\in \RR^d} f(x) := \frac{1}{n}\sum_{i=1}^{n}\max\{0,1-y_iw_i^Tx\} + \frac{\lambda}{2}\|x\|^2.$$
Here, a stochastic subgradient oracle can be given by selecting a summand $i\in[n]$ uniformly at random and then setting 
$$g_h(x;i)=\begin{cases} -y_iw_i & \text{ if } 1-y_iw_i^Tx\geq 0 \\ 0 & \text{ otherwise},\end{cases}$$
which has $L^2 = \frac{1}{n}\sum_{i=1}^n\|w_i\|^2$. 

Much work has previously been done to give guarantees for SVMs. If one adds the constraint that $x$ lies in some large ball $Q$ (which will then be projected onto at each iteration), the classic strongly convex rate can be applied~\cite{pegasos}. A similar approach utilized in~\cite{Lacoste-Julien2012} is to show that, in expectation, all of the iterates of a stochastic subgradient method lie in a large ball (provided the initial iterate does).
The specialized mirror descent method proposed by Lu~\cite{Lu2017} gives convergence guarantees for SVMs at a rate of $O(1/\sqrt{T})$ without needing a bounding ball. Splitting methods and quasi-Newton methods capable of solving this problem are given in~\cite{DuchiYoram2009} and~\cite{Yu2010}, respectively, which both avoid needing to assume subgradient bounds.

\subsection{Interpreting \eqref{eq:stochastic-generalization} as a Quadratic Growth Upper Bound}\label{subsec:quadratic-growth-interpretation}
Here we give an alternative interpretation of bounding the size of subgradients by \eqref{eq:stochastic-generalization} on some convex open set $U\subseteq\RR^d$ for deterministic subgradient oracles. In particular, suppose all $x\in U$ have
\begin{equation}\label{eq:deterministic-steepness-bound}
\|g(x)\|^2 \leq L^2_0 + L_1(f(x)-\inf_{x'\in U}f(x'))
\end{equation}
First consider the classic model where $L_1=0$. This is equivalent to $f$ being $L_0$-Lipschitz continuous on $U$ and can be restated as the following upper bound holding for each $x\in U$
$$ f(y) \leq f(x) + L_0\|y-x\| \hskip0.5cm \left(\forall y\in U\right).$$

This characterization shows the limitation to linear growth of the classic model. In the following proposition, we give a upper bound characterization when $L_1>0$, which can be viewed as allowing up to quadratic growth.

\begin{proposition} \label{prop:steepness-growth-bounds}
A convex function $f$ satisfies \eqref{eq:deterministic-steepness-bound} on some open convex $U\subseteq\RR^d$ if and only if the following quadratic upper bound holds for each $x\in U$
$$ f(y) \leq f(x) + \frac{L_1}{4}\|y-x\|^2 + \|y-x\|\sqrt{L_1(f(x)-\inf_{x'\in U}f(x')) + L_0^2} \hskip0.5cm \left(\forall y\in U\right).$$	
\end{proposition}
\begin{proof}
	First we prove the forward direction. Consider any $x,y\in U$ and subgradient oracle $g(\cdot)$. Let $v=(y-x)/\|y-x\|$ denote the unit direction from $x$ to $y$, and $h(t) = f(x+tv)-\inf_{x'\in U}f(x')$ denote the restriction of $f$ to this line shifted to have nonnegative value. Notice that $h(0)=f(x)-\inf_{x'\in U}f(x')$ and $h(\|y-x\|) = f(y)-\inf_{x'\in U}f(x')$. The convexity of $h$ implies it is differentiable almost everywhere in the interval $[0,\|y-x\|]$. Thus $h$ satisfies the following, for almost every $t\in[0,\|y-x\|]$,
	\begin{align*}
		|h'(t)| =|v^Tg(x+tv)| \leq \|g(x+tv)\|. %\leq L_1\delta(x+tv)+L_0
	\end{align*}
	
	This gives the differential inequality of
	$$ |h'(t)| \leq \sqrt{L_1h(t)+L^2_0}.$$	
	Standard calculus arguments imply
	$h(t) \leq  h(0) + \frac{L_1}{4}t^2 + t\sqrt{L_1 h(0) + L_0^2},$ which is equivalent to our claimed upper bound at $t=\|y-x\|$.
	
	Now we prove the reverse direction. Denote the upper bound given by some $x\in U$ as
	$$u_x(y):=f(x) + \frac{L_1}{4}\|y-x\|^2 + \|y-x\|\sqrt{L_1(f(x)-\inf_{x'\in U}f(x')) + L_0^2}.$$ Further, let $D_v$ denote the directional derivative operator in some unit direction $v\in \RR^d$. Then for any subgradient $g \in \partial f(x)$,
	$$v^Tg \leq D_v f(x) \leq D_v u_x(x),$$
	where the first inequality uses the definition of $D_v$ and the second uses the fact that $u_x$ upper bounds $f$. A simple calculation shows $D_v u_x(x) \leq \sqrt{L_1 (f(x)-\inf_{x'\in U}f(x')) +L_0^2}$. Then our subgradient bound follows by taking $v=g/\|g\|$.
\end{proof}

		\section{Convergence Proofs} \label{sec:proofs}
	Each of our extended convergence theorems follows from essentially the same proof as its classic counterpart.
	The central inequality in analyzing subgradient methods is the following.
	\begin{lemma} \label{lem:base-stochsatic-subgradient-inequality}
		Consider any convex function $f$. For any $x,y\in Q$ and $\alpha>0$,
		$$\EE_{\xi}\|P_Q(x - \alpha g(x;\xi)) - y\|^2 \leq \|x - y\|^2 - 2\alpha(\EE_{\xi}\ g(x;\xi))^T(x-y) + \alpha^2\EE_{\xi}\|g(x;\xi)\|^2.$$
	\end{lemma}
	\begin{proof}
		Since orthogonal projection onto a convex set is nonexpansive, we have
		\begin{align*}
		\|P_Q(x - \alpha g(x;\xi)) - y\|^2 & \leq \|x - \alpha g(x;\xi) - y\|^2\\
		& = \|x - y\|^2 - 2\alpha g(x;\xi)^T(x-y) + \alpha^2\|g(x;\xi)\|^2.
		\end{align*}
		Taking the expectation over $\xi\sim D$ yields
		$$ \EE_{\xi}\|P_Q(x - \alpha g(x;\xi)) - y\|^2 \leq \|x - y\|^2 - 2\alpha (\EE_{\xi}\ g(x;\xi))^T(x-y) + \alpha^2\EE_{\xi}\|g(x;\xi)\|^2.\qedhere$$
	\end{proof}

	Let $D_k = \EE_{\xi_{0 \dots T}} \|x_k-x^*\|$ denote the expected distance from each iterate to the minimizer $x^*$. Each of our proofs follows the same general outline: use Lemma~\ref{lem:base-stochsatic-subgradient-inequality} to set up a telescoping inequality on $D_k$, then sum the telescope. We begin by proving Shor's convergence result as its derivation is short and informative.

	\subsection{Proof of Shor's Theorem~\ref{thm:shor}}
	From Lemma~\ref{lem:base-stochsatic-subgradient-inequality} with $x=x_k$, $y = x^*$, and $\alpha=\alpha_k/\|g(x_k)\|$, it follows that 
	\begin{align*}
	D_{k+1}^2 \leq D_k^2 - \frac{2\alpha_k g(x_k)^T(x_k-x^*)}{\|g(x_k)\|} + \alpha_k^2,
	\end{align*}
	Inductively applying this implies
	$$0 \leq D_{T+1}^2 \leq D_0^2 - \sum_{k=0}^T 2\alpha_k\frac{g(x_k)^T(x_k-x^*)}{\|g(x_k)\|} + \sum_{k=0}^T \alpha_k^2.$$
	Thus
	$$\min_{k=0 \dots T}\left\{\frac{g(x_k)^T(x_k-x^*)}{\|g(x_k)\|}\right\} \leq \frac{D_0^2 + \sum_{k=0}^{T} \alpha_k^2}{2\sum_{k=0}^{T} \alpha_k},$$
	completing the proof.
	\hfill\ensuremath{\square}

	\subsection{Proof of Theorem~\ref{thm:extended-deterministic-subgradient-rate}}
	This follows directly from Theorem~\ref{thm:shor}. Note the result trivially holds if some $0\leq k\leq T$ has $x_k\in X^*$. Suppose $x_k$ satisfies the inequality in Theorem~\ref{thm:shor}. Let $y$ be the closest point in $\{x\ |\ g(x_k)^T(x-x_k)=0\}$ to $x^*$. Then our assumed growth bound implies
	$$f(y) - f^* \leq \mathcal{D}(\|y-x^*\|) \leq \mathcal{D}\left(\frac{D_0^2 + \sum_{k=0}^T \alpha_k^2}{2\sum_{k=0}^T \alpha_k}\right).$$
	The convexity of $f$ implies $f(x_k)\leq f(y)$ completing the proof. \hfill\ensuremath{\square}

	\subsection{Proof of Theorem~\ref{thm:extended-stochastic-subgradient-rate}}
	From Lemma~\ref{lem:base-stochsatic-subgradient-inequality} with $x=x_k$, $y = x^*$, and $\alpha=\alpha_k$, it follows that 
	\begin{align*}
		D_{k+1}^2 & \leq D_k^2 - \EE_{\xi_{0 \dots T}}\left[2\alpha_k (\EE_{\xi}g(x;\xi))^T(x_k-x^*)\right] + \alpha_k^2\EE_{\xi_{0 \dots T}}\|g(x_k,\xi_k)\|^2 \\
		& \leq D_k^2 - \EE_{\xi_{0 \dots T}}\left[(2\alpha_k -L_1\alpha_k^2)(f(x_k)-f^*)\right] + L_0^2\alpha_k^2,
	\end{align*}
	where the second inequality uses the convexity of $f$ and the bound on $\EE_{\xi}\|g(x;\xi)\|^2$. Inductively applying this implies
	$$0 \leq D_{T+1}^2 \leq D_0^2 - \EE_{\xi_{0 \dots T}}\left[\sum_{k=0}^T (2\alpha_k -L_1\alpha_k^2)(f(x_k)-f^*)\right] + L_0^2\sum_{k=0}^T \alpha_k^2.$$
	The convexity of $f$ gives
	$$\EE_{\xi_{0 \dots T}}\left[f\left(\frac{\sum_{k=0}^{T}\alpha_k(2-L_1\alpha_k)x_k}{\sum_{k=0}^{T}\alpha_k(2-L_1\alpha_k)}\right)-f^*\right] \leq \frac{D_0^2 + L_0^2\sum_{k=0}^{T} \alpha_k^2}{\sum_{k=0}^{T} \alpha_k(2-L_1\alpha_k)},$$
	completing the proof.
	\hfill\ensuremath{\square}

	\subsection{Proof of Theorem~\ref{thm:extended-strong-stochastic-subgradient-rate}}
	Our proof follows the style of~\cite{Lacoste-Julien2012}. Observe that our choice of step size $\alpha_k$ satisfies the following pair of conditions. First, note that it is a solution to the recurrence
	\begin{equation} \label{eq:first-stong-convexity}
	(k+1)\alpha_k^{-1}=(k+2)(\alpha^{-1}_{k+1}-\mu).
	\end{equation}
	Second, note that $L_1\alpha_k \leq 1$ for all $k\geq 0$ since
	\begin{equation} \label{eq:second-strong-convexity}
	L_1\alpha_k = \frac{2\mu(k+2) L_1}{(\mu(k+2))^2 + \frac{k+2}{k+1}L_1^2} \leq \frac{2\mu(k+2) L_1}{(\mu(k+2))^2 + L_1^2} \leq 1.
	\end{equation}
	
	From Lemma~\ref{lem:base-stochsatic-subgradient-inequality} with $x=x_k$, $y = x^*$, and $\alpha=\alpha_k$, it follows that 
	\begin{align*}
	D_{k+1}^2 & \leq D_k^2 - \EE_{\xi_{0 \dots T}}\left[2\alpha_k (\EE_{\xi}g(x_k;\xi))^T(x_k-x^*)\right] + \alpha_k^2\EE_{\xi_{0 \dots T}}\|g(x_k,\xi_k)\|^2 \\
	& \leq (1-\mu\alpha_k)D_k^2 - \EE_{\xi_{0 \dots T}}\left[(2\alpha_k -L_1\alpha_k^2)(f(x_k)-f^*)\right] + L_0^2\alpha_k^2,
	\end{align*}
	where the second inequality uses the strong convexity of $f$ and the bound on $\EE_{\xi}\|g(x;\xi)\|^2$. Multiplying by $(k+1)/\alpha_k$ yields
	\begin{align*}
	(k+1)\alpha_k^{-1}D_{k+1}^2 \leq & (k+1)(\alpha_k^{-1}-\mu)D_k^2 \\ &- \EE_{\xi_{0 \dots T}}\left[(k+1)(2 -L_1\alpha_k)(f(x_k)-f^*)\right] + L_0^2(k+1)\alpha_k.
	\end{align*}

	Notice that this inequality telescopes due to \eqref{eq:first-stong-convexity}. Inductively applying this implies
	\begin{equation*} 
	0 \leq (\alpha^{-1}_0-\mu) D_0^2-\EE_{\xi_{0 \dots T}}\left[\sum_{k=0}^T (k+1)(2 -L_1\alpha_k)(f(x_k)-f^*)\right] + L_0^2\sum_{k=0}^T (k+1)\alpha_k.
	\end{equation*}
	Since $\sum_{k=0}^{T} (k+1)\alpha_k \leq 2(T+1)/\mu$ and $\alpha^{-1}_0-\mu = L_1^2/2\mu$, we have
	\begin{equation*}
	\EE_{\xi_{0 \dots T}}\left[\sum_{k=0}^T (k+1)(2 -L_1\alpha_k)(f(x_k)-f^*)\right] \leq \frac{L_1^2D_0^2}{2\mu} + \frac{2L_0^2(T+1)}{\mu}.
	\end{equation*}
	Observe that the coefficients of each $f(x_k)-f^*$ above are positive due to \eqref{eq:second-strong-convexity}. Then the convexity of $f$ gives our first convergence bound.
	From \eqref{eq:second-strong-convexity}, we know $2-L_1\alpha_k\geq 1$ for all $k\geq 0$. Then the previous inequality can be weakened to 
	$$\EE_{\xi_{0 \dots T}}\left[\sum_{k=0}^T (k+1)(f(x_k)-f^*)\right] \leq \frac{L_1^2D_0^2}{2\mu} + \frac{2L_0^2(T+1)}{\mu}.$$
	The convexity of $f$ gives our second convergence bound.
	\hfill\ensuremath{\square}

	\section{Improved Convergence Without Strong Convexity}\label{sec:extensions}
Strong convexity is stronger than necessary to achieve many of the standard improvements in convergence rate for smooth optimization problems \cite{Bolte2017,Dima2016,Luo1993,Necoara-QuadraticGrowth}. Instead the weaker condition of requiring quadratic growth away from the set of minimizer suffices. We find that this weaker condition is also sufficient for \eqref{eq:Stochastic-Subgradient-Method} to have a convergence rate of $O(1/T)$.

We say a function $f$ has $\mu$-quadratic growth for some $\mu>0$ if all $x\in Q$ satisfy
\begin{equation*}
f(x) \geq f^* + \frac{\mu}{2}\dist(x,X^*)^2. \label{eq:quadratically-sharp}
\end{equation*}
The proof of Theorem \ref{thm:extended-strong-stochastic-subgradient-rate} only uses strong convexity once for the following inequality:
$$ g(x_k)^T(x_k-x^*) \geq f(x_k) - f^* + \frac{\mu}{2}\|x_k-x^*\|^2.$$
Having $\mu$-quadratic growth suffices to give a similar inequality, weakened by a factor of $1/2$:
$$ g(x_k)^T(x_k-P_{X^*}(x_k)) \geq f(x_k) - f^* \geq \frac{1}{2}\left(f(x_k) - f^*\right) + \frac{\mu}{4}\dist(x_k,X^*)^2.$$
Then simple modifications of the proof of Theorem~\ref{thm:extended-strong-stochastic-subgradient-rate} give a $O(1/T)$ convergence rate.

\begin{theorem}
	Consider any convex function $f$ with $\mu$-quadratic growth and stochastic subgradient oracle satisfying \eqref{eq:stochastic-generalization}. Then for the decreasing sequence of step sizes $$\alpha_k = \frac{4}{\mu(k+2) + \frac{4L_1^2}{\mu(k+1)}},$$ the iteration \eqref{eq:Stochastic-Subgradient-Method} satisfies
	\begin{equation*}
	\EE_{\xi_{0 \dots T}}\left[f\left(\frac{\sum_{k=0}^{T}(k+1)(1-L_1\alpha_k)x_k}{\sum_{k=0}^{T}(k+1)(1-L_1\alpha_k)}\right)-f^*\right]  \leq \frac{4L_0^2(T+1) + L_1^2\dist(x_0,X^*)^2}{\mu\sum_{k=0}^{T}(k+1)(1-L_1\alpha_k)}.
	\end{equation*}
\end{theorem}
\begin{proof}
	Observe that our choice of step size $\alpha_k$ satisfies the following pair of conditions. First, note that it is a solution to the recurrence
	\begin{equation} \label{eq:first-quadratic-growth}
	(k+1)\alpha_k^{-1}=(k+2)(\alpha^{-1}_{k+1}-\mu/2).
	\end{equation}
	Second, note that $L_1\alpha_k < 1$ for all $k\geq 0$. This follows as
	\begin{equation} \label{eq:second-quadratic-growth}
	L_1\alpha_k = \frac{4\mu(k+2) L_1}{(\mu(k+2))^2 + 4\frac{k+2}{k+1}L_1^2} \leq \frac{4\mu(k+2) L_1}{(\mu(k+2))^2 + (2L_1)^2} \leq 1,
	\end{equation}
	where the first inequality is strict if $L_1>0$ and the second inequality is strict if $L_1=0$.
	
	Let $D_k = \EE_{\xi_{0 \dots T}} \dist(x_k,X^*)$ denote the expected distance from each iterate to the set of minimizers, $X^*$. From Lemma~\ref{lem:base-stochsatic-subgradient-inequality} with $x=x_k$, $y = P_{X^*}(x_k)$, and $\alpha=\alpha_k$, it follows that 
	\begin{align*}
	D_{k+1}^2 & \leq D_k^2 - \EE_{\xi_{0 \dots T}}\left[2\alpha_k (\EE_{\xi}g(x_k;\xi))^T(x_k-x^*)\right] + \alpha_k^2\EE_{\xi_{0 \dots T}}\|g(x_k,\xi_k)\|^2 \\
	& \leq (1-\mu\alpha_k/2)D_k^2 - \EE_{\xi_{0 \dots T}}\left[(\alpha_k -L_1\alpha_k^2)(f(x_k)-f^*)\right] + L_0^2\alpha_k^2,
	\end{align*}
	where the second inequality uses the quadratic growth of $f$ and the bound on $\EE_{\xi}\|g(x;\xi)\|^2$. Multiplying by $(k+1)/\alpha_k$ yields
	\begin{align*}
	(k+1)\alpha_k^{-1}D_{k+1}^2 \leq & (k+1)(\alpha_k^{-1}-\mu/2)D_k^2 \\ & - \EE_{\xi_{0 \dots T}}\left[(k+1)(1 -L_1\alpha_k)(f(x_k)-f^*)\right] + L_0^2(k+1)\alpha_k.
	\end{align*}

	Notice that this inequality telescopes due to \eqref{eq:first-quadratic-growth}. Inductively applying this implies
	\begin{equation*} 
	0 \leq (\alpha^{-1}_0-\mu/2) D_0^2-\EE_{\xi_{0 \dots T}}\left[\sum_{k=0}^T (k+1)(1 -L_1\alpha_k)(f(x_k)-f^*)\right] + L_0^2\sum_{k=0}^T (k+1)\alpha_k.
	\end{equation*}
	Since $\sum_{k=0}^{T} (k+1)\alpha_k \leq 4(T+1)/\mu$ and $\alpha^{-1}_0-\mu/2 = L_1^2/\mu$, we have
	\begin{equation*}
	\EE_{\xi_{0 \dots T}}\left[\sum_{k=0}^T (k+1)(1 -L_1\alpha_k)(f(x_k)-f^*)\right] \leq \frac{L_1^2D_0^2}{\mu} + \frac{4L_0^2(T+1)}{\mu}.
	\end{equation*}
	Observe that the coefficients of each $f(x_k)-f^*$ above are positive due to \eqref{eq:second-quadratic-growth}. Then the convexity of $f$ completes the proof.
\end{proof}

	\paragraph{Acknowledgments.} The author thanks Jim Renegar for providing feedback on an early draft of this work.
	
	\bibliographystyle{plain}
	{\small \bibliography{references}}

\begin{thebibliography}{10}

\bibitem{Bolte2017}
J{\'e}r{\^o}me Bolte, Trong~Phong Nguyen, Juan Peypouquet, and Bruce~W. Suter.
\newblock From error bounds to the complexity of first-order descent methods
  for convex functions.
\newblock {\em Mathematical Programming}, 165(2):471--507, Oct 2017.

\bibitem{Bubeck-textbook}
S{\'e}bastien Bubeck.
\newblock {Convex Optimization: Algorithms and Complexity}.
\newblock {\em Found. Trends Mach. Learn.}, 8(3-4):231--357, November 2015.

\bibitem{DavisGrimmer2017}
Damek Davis and Benjamin Grimmer.
\newblock {Proximally Guided Stochastic Subgradient Method for Nonsmooth,
  Nonconvex Problems}.
\newblock {\em ArXiv e-prints}, 1707.03505, July 2017.

\bibitem{Dima2016}
Dima {Drusvyatskiy} and Adrian {Lewis}.
\newblock {Error bounds, quadratic growth, and linear convergence of proximal
  methods}.
\newblock {\em ArXiv e-prints}, February 2016.

\bibitem{DuchiYoram2009}
John Duchi and Yoram Singer.
\newblock {Efficient Online and Batch Learning Using Forward Backward
  Splitting}.
\newblock {\em J. Mach. Learn. Res.}, 10:2899--2934, December 2009.

\bibitem{Grimmer2017-radial-subgradient}
Benjamin Grimmer.
\newblock {Radial Subgradient Method}.
\newblock {\em SIAM Journal on Optimization}, 28(1):459--469, 2018.

\bibitem{HazanKale2014}
Elad Hazan and Satyen Kale.
\newblock {Beyond the Regret Minimization Barrier: Optimal Algorithms for
  Stochastic Strongly-convex Optimization}.
\newblock {\em J. Mach. Learn. Res.}, 15(1):2489--2512, January 2014.

\bibitem{Lacoste-Julien2012}
Simon {Lacoste-Julien}, Mark {Schmidt}, and Francis {Bach}.
\newblock {A simpler approach to obtaining an O(1/t) convergence rate for the
  projected stochastic subgradient method}.
\newblock {\em ArXiv e-prints}, 1212.2002, December 2012.

\bibitem{Lu2017}
Haihao Lu.
\newblock {``Relative-Continuity'' for Non-Lipschitz Non-Smooth Convex
  Optimization using Stochastic (or Deterministic) Mirror Descent}.
\newblock {\em ArXiv e-prints}, 1710.04718, October 2017.

\bibitem{Luo1993}
Zhi-Quan Luo and Paul Tseng.
\newblock Error bounds and convergence analysis of feasible descent methods: a
  general approach.
\newblock {\em Annals of Operations Research}, 46(1):157--178, Mar 1993.

\bibitem{Necoara-QuadraticGrowth}
Ion {Necoara}, Yurii {Nesterov}, and Francois {Glineur}.
\newblock {Linear convergence of first order methods for non-strongly convex
  optimization}.
\newblock {\em To appear in Mathematical Programming}.

\bibitem{Nedic2014}
Angelia Nedi\'c and Soomin Lee.
\newblock On stochastic subgradient mirror-descent algorithm with weighted
  averaging.
\newblock {\em SIAM Journal on Optimization}, 24(1):84--107, 2014.

\bibitem{Nesterov-introductory}
Yurii Nesterov.
\newblock {\em {Introductory Lectures on Convex Optimization: A Basic Course}}.
\newblock Springer Publishing Company, Incorporated, 1 edition, 2004.

\bibitem{Rakhlin2012}
Alexander Rakhlin, Ohad Shamir, and Karthik Sridharan.
\newblock {Making Gradient Descent Optimal for Strongly Convex Stochastic
  Optimization}.
\newblock In {\em Proceedings of the 29th International Coference on
  International Conference on Machine Learning}, ICML'12, pages 1571--1578,
  USA, 2012. Omnipress.

\bibitem{Renegar2016}
James Renegar.
\newblock {``Efficient'' Subgradient Methods for General Convex Optimization}.
\newblock {\em SIAM Journal on Optimization}, 26(4):2649--2676, 2016.

\bibitem{pegasos}
Shai Shalev-Shwartz, Yoram Singer, Nathan Srebro, and Andrew Cotter.
\newblock Pegasos: primal estimated sub-gradient solver for svm.
\newblock {\em Mathematical Programming}, 127(1):3--30, Mar 2011.

\bibitem{Shor1985}
Naun~Zuselevich Shor.
\newblock {\em Minimization Methods for Non-Differentiable Functions}, page~23.
\newblock Springer Berlin Heidelberg, Berlin, Heidelberg, 1985.

\bibitem{Yu2010}
Jin Yu, S.V.N. Vishwanathan, Simon G\"{u}nter, and Nicol~N. Schraudolph.
\newblock {A Quasi-Newton Approach to Nonsmooth Convex Optimization Problems in
  Machine Learning}.
\newblock {\em J. Mach. Learn. Res.}, 11:1145--1200, March 2010.

\end{thebibliography}

\end{document}